\documentclass[10pt]{amsart}
\usepackage{amsmath, amsfonts, amsthm, amssymb}
\usepackage[margin=1in]{geometry}
\newtheorem{thm}{Theorem}[section]
\newtheorem{cor}[thm]{Corollary}
\newtheorem{lem}[thm]{Lemma}
\newtheorem{prop}[thm]{Proposition}
\newtheorem{rem}[thm]{Remark}
\numberwithin{equation}{section}
\begin{document}

\title{Inverse problem for \text{P}ell equation and real quadratic fields of the least type}
\thanks{Supported by the National Research Foundation of Korea(NRF) grant funded by the Korea government(MEST) (2010-0026473).}

\author{ Park, Jeongho }

\address{Department of Mathematics (Room 117),
Pohang University of Science and Technology,
San 31 Hyoja Dong, Nam-Gu, Pohang 790-784, KOREA.
Tel. 82-10-3047-7793.}

\email{pkskng@postech.ac.kr}
%\keywords{Real quadratic field, Class number, Fundamental unit, Principal ideal, Prime ideal}
%\subjclass{Primary:11R29, Secondary: 11R11, 11J68, 11Y40}
%11R29: Class numbers, class groups, discriminants
%11Y40: Algebraic number theory computations
%11J68: Approximation to algebraic numbers
%11R11: Quadratic extensions
\thanks{Supported by the National Research Foundation of Korea(NRF) grant funded by the Korea government(MEST)
(2010-0026473).}

\maketitle

\begin{abstract}
The purpose of this article is to give the solutions of the inverse problem for Pellian equations. For any rational number $0< a/b < 1$, the fundamental discriminants $D$ satisfying $(\lfloor \sqrt{D} \rfloor b + a)^2 - D b^2 = 4$ are given in terms of a quadratic progression. There were studies about this problem based on symmetric sequences $\{a_1,\cdots,a_{l-1}\}$ and periodic continued fractions $[a_0,\overline{ a_1,\cdots,a_{l-1},a_l }]$, but in this article we solve the problem in a completely different way with simpler parameters. The result is obtained by measuring the quality of approximation of a rational number to $\sqrt{d}$ or $\frac{1+\sqrt{d}}{2}$, and by defining a short interval attached to each rational number. On this formulation we also show that for almost all square-free integer $d$, $d$ is the least element of the prescribed quadratic progression for some $a/b$.
\keywords{Real quadratic field, Continued fraction, Pell equation, Symmetric sequence, Least type}
%\subclass{Primary:11J68, Secondary: 11R29, 11Y40}
\subjclass{Primary:11J68, Secondary: 11R29, 11Y40}
\end{abstract}

\section{Introduction}\label{sec_intro}

Let $d$ be a non-square positive integer, and
\begin{gather*}
    D = \begin{cases}
d & \text{if $d \equiv 1$ mod $4$},\\
4d& \text{otherwise}
\end{cases}
\quad
\omega_d =\begin{cases}
\frac{1 + \sqrt{d}}{2} & \text{if $d \equiv 1$ mod $4$}\\
\sqrt{d} & \text{otherwise}.
\end{cases}
\end{gather*}
By Dirichlet's unit theorem, the set of positive solutions to the Pell's equation $X^2 - D Y^2 = 4$ forms a cyclic group generated by the fundamental unit $\varepsilon_d > 1$. Consider the continued fraction expansion $\omega_d = [a_0,a_1,a_2,\cdots]$. It is well known that the expansion of $\omega_d$ is of the form $\omega_d = [a_0,\overline{a_1,\cdots,a_{l-1}, a_l}]$ where $\{ a_1,\cdots,a_{l-1} \}$ is a symmetric sequence and $l = l(\omega_d)$ is the (minimal) period of the expansion  \cite{Halter_Koch}. By proposition~\ref{prop_partial quotients of omega d} $\varepsilon_d$ is of the form $\varepsilon_d = X + Y\omega_d$ or $X - Y + Y\omega_d$ where $X/Y = [a_0,a_1,\cdots,a_{l-1}]$. Here $a_0 = \lfloor \omega_d \rfloor$ is readily determined for any given $d$. Therefore the \emph{direct} problem, i.e., to find $\varepsilon_d$, is essentially the same as the following: given $d$, find corresponding symmetric sequence $\{ a_1,\cdots,a_{l-1} \}$.

In the direct problem, the length of the symmetric sequence is of great interest but we know very little about it. Assume $d$ is square-free so that $D$ is the field discriminant of $\mathbb{Q}(\sqrt{d})$. Dirichlet's class number formula for $\mathbb{Q}(\sqrt{d})$ is
\begin{equation}\label{eqn_Dirichlet class number formula}
    h_d = \frac{\sqrt{D}L(1,\chi)}{2 \log \varepsilon_d}
\end{equation}
where the $L$-value is bounded in a relatively narrow range $\frac{1}{D^{\epsilon}} \ll L(1,\chi) \ll \log D$  for any $\epsilon > 0$ \cite{Li}. It is a fact that the class number $h_d$ varies in a very wide range $1 \leq h_d \ll \sqrt{D}$, which is necessarily equivalent to saying that $\log \varepsilon_d$ varies as much as this. Using $q_n = a_n q_{n-1} + q_{n-2}$ \cite{Hardy}, $a_n < \omega_d + 1/2$ for $n < l$(proposition~\ref{prop_partial quotients of omega d}) and $\frac{p_{l-1}}{q_{l-1}} - \omega_d \ll (q_{l-1})^{-2}$, it can be shown that $l(\omega_d) \ll \log \varepsilon_d \ll l(\omega_d)\omega_d$. Therefore the size of $h_d$ is closely related to the length $l(\omega_d)$ of the symmetric sequence. It is believed that the period $l(\omega_d)$ of $\omega_d$ is as large as $D^{1/2 - \epsilon}$ fairly often, but there is absolutely no result even close to this according to the author's knowledge.

On the other hand, the \emph{inverse} problem is as follows: given a symmetric sequence $\{ a_1,\cdots,a_{l-1} \}$ of positive integers, find all $d$'s such that $\omega_d = [a_0,\overline{a_1,\cdots,a_{l-1}, a_l}]$ for some $a_0,a_l$. Let $\frak{D}'(a_1,\cdots,a_{l-1})$ be the set of positive non-square integers $d \equiv 1$ mod $4$ such that $\omega_d = [a_0,\overline{a_1,\cdots,a_{l-1},a_l}]$ for some $a_0,a_l$. In \cite{Friesen},\cite{Halter_Koch}, a polynomial $f$ of degree 2 is associated to $\{ a_1,\cdots,a_{l-1} \}$ so that either $\frak{D}'(a_1,\cdots,a_{l-1})$ is an empty set or it consists of all positive integers of the form $d = f(m)$ with $m \geq m_0$, where $m_0$ depends on the symmetric sequence. Similarly, $\frak{D}(a_1,\cdots,a_{l-1})$ denotes the set of positive non-square integers $d$ such that $\sqrt{d} = [a_0,\overline{a_1,\cdots,a_{l-1},a_l}]$, which either is empty or consists of all positive integers of the form $d = f(m)/4$ with $m \geq m_0$.

Like the length $l(\omega_d)$ is of interest in the direct problem, the size of $d$ is of interest in the inverse problem. It seems that the smallest elements of $\frak{D}(a_1,\cdots,a_n)$ and $\frak{D}'(a_1,\cdots,a_n)$ are very different from the remaining elements, as some studies show. For example, let $p$ be a prime $\equiv 1$ mod $4$ and $(t + u\sqrt{p})/2$ the fundamental unit of $\mathbb{Q}(\sqrt{p})$. Ankeny, Artin and Chowla conjectured that \cite{AnkenyArtinChowla} $u \not\equiv 0$ mod $p$, and Hashimoto showed that \cite{Hashimoto2001143} unless $p$ is the smallest element of $\frak{D}'(a_1,\cdots,a_{l-1})$ for some $\{ a_1,\cdots,a_{l-1} \}$, $u$ is less than $p$ and so the conjecture is true for this $p$. Another example is the notion of \emph{minimal type} introduced in \cite{Kawamoto}. Throughout section 3 of \cite{Kawamoto}, it can be red off that whenever a non-square integer $d$ is of minimal type for $\sqrt{d}$ or $(1+\sqrt{d})/2$, it is the smallest element of $\frak{D}(a_1,\cdots,a_{l-1})$ or $\frak{D}'(a_1,\cdots,a_{l-1})$ for some symmetric sequence $\{ a_1,\cdots,a_{l-1}\}$. In that paper, it is shown that fundamental units of real quadratic fields that are not of minimal type are relatively small, and among such fields exactly 51 (with one more possible exception) have class number 1.

The purpose of this paper is to treat the inverse problem in a different way, and to show that almost all non-square integers are the least elements of $\frak{D}(a_1,\cdots,a_{l-1})$ or $\frak{D}'(a_1,\cdots,a_{l-1})$ for some $\{ a_1,\cdots,a_{l-1}\}$. In this paper we say `almost all' to mean
\begin{equation*}
    \lim _{N \rightarrow \infty}\frac{\#\{ \text{exceptions between } 1 \text{ and } N \}}{N} = 0.
\end{equation*}
Write $x/y = [0,a_1,\cdots,a_{l-1}]$ so that the direct problem becomes to find corresponding rational number $x/y$ for a given $d$. We forget about the symmetry of $\{a_1,\cdots,a_{l-1}\}$ now, and formulate the inverse problem as follows: given any nonnegative rational number $x/y < 1$, find all $d$'s such that $\varepsilon_d = a_0 y  + x + y \omega_d$ or $(a_0 - 1) y  + x + y \omega_d$. The basic idea in our approach is to examine the approximation quality of $x/y$ to the fractional part of $\omega_d$. Roughly speaking, we will show that this approximation can be `sufficiently good' only when $x^2 \equiv \pm 1$ mod $y$ and $d = g(m)$ for some integer $m$, where $g$ is a quadratic polynomial that depends on $x/y$. Consequently we rediscover quadratic progressions related to the inverse problem.

The contents are as follows. In section~\ref{sec_Preliminary results} we list down several facts about square-free integers and continued fractions, together with prescribed results from \cite{Friesen},\cite{Halter_Koch}. In section~\ref{sec_Attached intervals} we will explain what the meaning of `sufficiently good' shall be, and define very narrow intervals assigned to each positive rational number $a_0 + x/y$. We determine exactly when such an interval contains an integer, and specifies that integer in terms of a quadratic polynomial. In section~\ref{sec_Dominance of the least elements} we prove that the non-square integers that are not the smallest elements of $\frak{D}(a_1,\cdots,a_{l-1})$ or $\frak{D}'(a_1,\cdots,a_{l-1})$ constitute a measure zero set among natural numbers, showing that almost all real quadratic number fields are of the `least type'.

% -----------------------------------------------------------
\section{Preliminary results}\label{sec_Preliminary results}

Let $Q(x)$ be the number of square-free integers between 1 and $x$. It is well known(for example, see theorem 333 of \cite{Hardy}) that $Q(x) = \frac{6}{\pi^2}x + O(\sqrt{x})$. In \cite{Chao_Hua}, under Riemann hypothesis it was proved that $Q(x) = \frac{6}{\pi^2}x + O(x^{17/54 + \epsilon})$. This suggests that about $60.79\%$ of the integers between $n^2$ and $(n+1)^2$ shall be square-free for every sufficiently large $n$. In addition to this, there is a useful

\begin{thm}[\cite{Prachar}\label{thm_square-free_congruence}]
Let $S(x;c,k)$ be the number of square-free integers between $1$ and $x$ and that are congruent to $c$ modulo $k$. Assume $(c,k)=1$ and $k \leq x^{2/3-\epsilon}$. Then
\begin{equation*}
S(x;c,k) \sim \frac{6 x}{\pi^2 k} \prod_{p \mid k} \left(  1 - \frac{1}{p^2}  \right)^{-1} \;\;\;\;(x \rightarrow \infty).
\end{equation*}
\end{thm}

In particular, $S(x;1,4) \sim S(x;3,4) \sim \frac{1}{3} \frac{6 }{\pi^2}x$ and hence each third of square-free numbers is congruent to $1,2$, and $3$ modulo $4$.\\

Regarding continued fractions, we mostly use the conventions in chapter 10 of \cite{Hardy}. We denote the simple continued fraction expansion of a positive real number $x$ by $x = [a_0,a_1,a_2,\cdots]$ and its $n$-th convergent by $p_n / q_n = [a_0,a_1,\cdots,a_n]$. By definition, we always assume $a_i > 0$ for $i > 0$. We write $[0,a_1,\cdots,a_n] = r_n / q_n$ and use the convention $(q_{-2},p_{-2}) = (1,0)$, $(q_{-1},p_{-1}) = (0,1)$. Given an expansion of $x$, we call $a_n$ the $n$-th partial quotient of $x$, and $\alpha_n = [a_n,a_{n+1},a_{n+2},\cdots]$ the $n$-th total quotient. It worths to mention that $q_n$ is determined only by $a_1,a_2,\cdots,a_n$ because $a_0 = \lfloor p_n / q_n \rfloor$ has nothing to do with the denominator $q_n$.

In this manuscript we consider non-square integers $d$ and the expansion $\omega_d = [a_0,a_1,a_2,$ $\cdots]$. For $x \in \mathbb{Q}(\sqrt{d})$, let $\overline{x}$ be its conjugate and $N(x) = x \overline{x}$. For the $n$-th convergent $p_n / q_n$ of $\omega_d$, put

\begin{equation}\label{eqn_xi_n}
\xi_n = \overline{ p_n - q_n \omega_d} =\begin{cases}
p_n - q_n + q_n \omega_d & \text{if $d \equiv 1$ mod $4$}\\
p_n + q_n \omega_d & \text{otherwise}
\end{cases}
\end{equation}

and let $\nu_n = |N(\xi_n)| = |\frac{p_n}{q_n} - \omega_d| q_n \xi_n$. Recall that $p_n / q_n > \omega_d$ if and only if $n$ is odd (theorem 163 of \cite{Hardy}), so
\begin{equation}\label{eqn_nu_n}
    \nu_n = (-1)^{n+1}N(\xi_n).
\end{equation}
We say that a quadratic integer $\xi \in \mathbb{Z}[\omega_d]$ \emph{comes from} a convergent to $\omega_d$ when $\xi = \xi_n$ for some $n$. A quadratic integer with norm $\pm 1$ is called a quadratic unit. As usual $\lfloor x \rfloor$ denotes the greatest integer not exceeding $x$.

On this setting, we have a basic
\begin{prop}[Theorem 150 in \cite{Hardy}]\label{prop_p q minus q p}
$p_n q_{n-1} - p_{n-1}q_n = (-1)^{n+1}$.
\end{prop}

The following appears on p.141 of \cite{Hardy}:
\begin{equation}\label{eqn_total quotient to number}
    [a_0,a_1,\cdots,a_n,\alpha_{n+1}] = \frac{\alpha_{n+1} p_n + p_{n-1}}{\alpha_{n+1} q_n + q_{n-1}}
\end{equation}
and hence

\begin{equation}\label{eqn_difference of convergents}
    [a_0,\cdots,a_n,\alpha_{n+1}] - [a_0,\cdots,a_n] = \frac{\alpha_{n+1} p_n + p_{n-1}}{\alpha_{n+1} q_n + q_{n-1}} - \frac{p_n}{q_n} = \frac{(-1)^n}{q_n(\alpha_{n+1} q_n + q_{n-1})}.
\end{equation}

We also include a

\begin{lem}\label{lem_quotient_norm}
For $n \geq 0$
\begin{equation*}
\alpha_{n+1} =
\frac{\sqrt{D}}{\nu_n} - \frac{q_{n-1}}{q_n} + \delta_n < \frac{\sqrt{D}}{\nu_n}
\end{equation*}
where $|\delta_n| < \frac{4}{q^2_n \sqrt{D}}$.
\end{lem}

\begin{proof}
We first prove that $\alpha_{n+1} < \sqrt{D}/\nu_n$ for $n \geq 0$. There are four fundamental discriminants 5,8,12,13 that are less than 16, and we have $\omega_2 = [1,\overline{2}]$, $\omega_3 = [1,\overline{1,2}]$, $\omega_5 = [1,\overline{1}]$, $\omega_{13} = [2,\overline{3}]$. One can easily check that $\alpha_{n+1} < \sqrt{D}/\nu_n$ for $n \geq 0$ in all these cases, so we assume $D>16$. Now suppose we have proved $|\delta_n| < \frac{4}{q^2_n \sqrt{D}}$. For $n \geq 1$, since $q_{n-1} > 0$, the term $q_{n-1}/q_n - \delta_n$ is positive and $\alpha_{n+1} < \sqrt{D}/\nu_n$. When $n = 0$, write
\begin{gather*}
\alpha_1 = \frac{1}{\omega_d - \lfloor \omega_d \rfloor}, \quad \xi_0 = \overline{p_0 - q_0 \omega_d} = \overline{\lfloor \omega_d \rfloor - \omega_d}, \\
\nu_0 = \begin{cases}
-(\lfloor \omega_d \rfloor - \omega_d)(\lfloor \omega_d \rfloor -1 + \omega_d)\;\; &\text{ if $d \equiv 1$ mod 4}\\
-(\lfloor \omega_d \rfloor - \omega_d)(\lfloor \omega_d \rfloor + \omega_d)\;\; &\text{ otherwise}
\end{cases}
\end{gather*}
and so
\begin{equation*}
    \alpha_1 = \begin{cases}
    \frac{\lfloor \omega_d \rfloor -1 + \omega_d}{\nu_0} < \frac{\frac{1+\sqrt{d}}{2} + \frac{-1+\sqrt{d}}{2}}{\nu_0} = \frac{\sqrt{D}}{\nu_0}\;\; &\text{ if $d \equiv 1$ mod 4}\\
    \frac{\lfloor \omega_d \rfloor + \omega_d}{\nu_0} < \frac{\sqrt{d} + \sqrt{d}}{\nu_0}= \frac{\sqrt{D}}{\nu_0}\;\; &\text{ otherwise}.
    \end{cases}
\end{equation*}
Thus it suffices to prove $|\delta_n| < \frac{4}{q^2_n \sqrt{D}}$.

The cases $d \equiv$ 2 and 3 (mod 4) are easier in computation, so here we assume $d \equiv 1$ (mod 4) so that $\omega_d = \frac{1 + \sqrt{d}}{2}$.
Recall that the continued fraction expansion of $\omega_d$ has a natural geometric interpretation on $xy$-plane. Let $O=(0,0)$ be the origin of the $xy$-plane, $A = (q_{n-1}, p_{n-1})$, $B = (q_{n}, p_{n})$, $C$ the intersection of $\overline{AB}$ and the line $y = \omega_d x$, and $D = (q_n, \omega_d q_n)$. Then $[\overline{AC}:\overline{CB}] = [\alpha_{n+1}:1]$ and the area of $\triangle{OAB}$ is 1/2. Observe that the area of $\triangle{OBD}$ is $\frac{1}{2}|(p_n - q_n \omega_d)q_n|$. Let $B' = (0,p_n)$, $D' = (0,\omega_d q_n)$.

We have
\begin{equation*}
\frac{\xi_n \overline{\xi_n}}{q_n^2} = \left( \frac{p_n}{q_n} - 1 + \omega_d \right) \left( \frac{p_n}{q_n} - 1 + 1 - \omega_d \right) = \pm \frac{\nu_n}{q_n^2}
\end{equation*}
or
\begin{equation*}
\frac{p_n}{q_n} - \omega_d = \frac{\pm \nu_n}{q_n (p_n - q_n + \omega_d q_n) }
\end{equation*}
and therefore
\begin{align*}
|\square B'B D D'| &= |(p_n - q_n \omega_d) q_n |\\
&= \frac{\nu_n}{p_n / q_n - 1 + \omega_d}\\
&= \frac{\nu_n}{2 \omega_d - 1 + (-1)^{n+1} \frac{\nu_n}{q_n (p_n - q_n + \omega_d q_n)}}\\
&= \frac{\nu_n}{2 \omega_d - 1} \left( \frac{1}{1 + (-1)^{n+1} \frac{\nu_n}{(2 \omega_d - 1) q_n (p_n - q_n + \omega_d q_n)} }  \right)\\
&= \frac{\nu_n}{\sqrt{d}} \left( 1 + \epsilon_n \right)
\end{align*}

where $\epsilon_n = \frac{(-1)^{n+1} \nu_n}{\sqrt{d} q_n (p_n - q_n + \omega_d q_n)} < \frac{1}{\omega_d q_n^2} \frac{\nu_n}{\sqrt{d}}$. Examining the ratios of the coordinates of $A,B$ and $C$, it is easily deduced that the area of $\triangle{BCD}$ is $\frac{1 - q_{n-1}/q_n}{1 + \alpha_{n+1}} \frac{\nu_n}{2\sqrt{d}}\left( 1+ \epsilon_n  \right) $, and hence
\begin{align*}
|\triangle{OBC}| & = |\triangle{OBD}| - |\triangle{BCD}| \\
&= \left(1 - \frac{1 - q_{n-1}/q_n}{1 + \alpha_{n+1}} \right) \frac{\nu_n}{2\sqrt{d}}\left( 1+ \epsilon_n  \right)\\
&= \left(\frac{\alpha_{n+1} + q_{n-1}/q_n}{1 + \alpha_{n+1}} \right) \frac{\nu_n}{2\sqrt{d}}\left( 1+ \epsilon_n  \right)
\end{align*}
But $|\triangle{OBC}| = |\triangle{OAB}| \frac{1}{1+\alpha_{n+1}} = \frac{1}{2(1+\alpha_{n+1})}$, whence $\left( \alpha_{n+1} + \frac{q_{n-1}}{q_n} \right) \frac{\nu_n}{\sqrt{d}}( 1+ \epsilon_n ) = 1$. Thus $\alpha_{n+1} = \frac{\sqrt{d}}{\nu_n} (1 + \epsilon_n') - \frac{q_{n-1}}{q_n}$ where $\epsilon_n' = -\epsilon_n + \epsilon_n^2 - \epsilon_n^3 + \cdots$ so $|\epsilon_n'| < 2 |\epsilon_n| < \frac{2\nu_n}{\omega_d q_n^2 \sqrt{d}}$. It follows that
\begin{equation*}
    \alpha_{n+1} = \frac{\sqrt{d}}{\nu_n} - \frac{q_{n-1}}{q_n} + \epsilon_n''
\end{equation*}
where $|\epsilon_n''| < \frac{2}{q_n^2 \omega_d} < \frac{4}{q_n^2 \sqrt{D}}$, which proves the lemma in case $d \equiv 1$ (mod 4).

When $d \equiv$ 2 or 3 (mod 4), exactly the same computation with continued fraction of $\omega_d = \sqrt{d}$ completes the proof.
\end{proof}

The followings are well known facts which we quote in appropriate forms.

\begin{prop}[Theorem 162 of \cite{Hardy}]\label{prop_two ways of continued fractions}
A positive rational number which is not equal to 1 can be expressed as a finite simple continued fraction in exactly two ways, one with an even and the other with an odd number of convergents. In one form the last partial quotient is 1, in the other it is greater than 1.
\end{prop}

For a rational number $p/q \neq 1$, we will usually write
\begin{equation*}
    p/q = [a_0,\cdots,a_n,1] = [a_0,\cdots,a_{n-1},1 + a_n].
\end{equation*}

\begin{prop}[\cite{Hardy}\label{prop_Hardy_convergent}]
If (p,q) = 1 and
\begin{equation*}
\left| \frac{p}{q} - x \right| < \frac{1}{2 q^2}
\end{equation*}
then $p/q$ is a convergent to $x$.
\end{prop}

We also recall that a quadratic irrational $x$ is \emph{reduced} if $x > 1$ and $-1 < \overline{x} < 0$, and that the continued fraction expansion of $x$ is purely periodic if and only if $x$ is reduced (for example, see theorem 7.20 of \cite{niven1991introduction}). In particular $\sqrt{d} + \lfloor \sqrt{d} \rfloor$ and $\frac{1 + \sqrt{d}}{2} + \lfloor \frac{1 + \sqrt{d}}{2} \rfloor - 1$ are reduced, so the continued fraction expansion of $\omega_d$ is of the form
\begin{gather}\label{eqn_expansion of omega d}
    \sqrt{d} = [a_0,\overline{a_1,\cdots,a_{l-1}, 2 a_0}],\quad    \frac{1+\sqrt{d}}{2} = [a_0,\overline{a_1,\cdots,a_{l-1}, 2 a_0 - 1}].
\end{gather}

\begin{prop}\label{prop_partial quotients of omega d}
Let $d$ be a non-square positive integer and assume
\begin{equation*}
    \sqrt{d} = [a_0,\overline{a_1,\cdots,a_{l-1}, 2 a_0}]
\end{equation*}
where $l$ is the period of this expansion. Then $a_i < \sqrt{d} + 1$ for $0 < i < l$, and the fundamental unit of $\mathbb{Z}[\sqrt{d}]$ comes from the $(l-1)$-th convergent to $\sqrt{d}$. Similarly, let $d \equiv 1$ mod $4$ be a non-square positive integer and let $\omega_d = [b_0,\overline{b_1,\cdots,b_{l'-1}, 2 b_0 - 1}]$. Then $b_j < \omega_d + 1$ for $0 < j < l'$ and the fundamental unit of $\mathbb{Z}[\omega_d]$ comes from the $(l'-1)$-th convergent to $\omega_d$.
\end{prop}
This fact has been seen several times in the literature, but the author couldn't find a compact proof so we include an elementary one here.
\begin{proof}
For brevity, let $\omega$ be one of $\sqrt{d}$ and $\frac{1 + \sqrt{d}}{2}$. ($\omega = \sqrt{d}$ is allowed for $d \equiv 1$ mod $4$ too). Write $\omega = [a_0,a_1,a_2,\cdots]$ and let $l(\omega)$ be its period.

Referring to a table of continued fractions of $\sqrt{d}$ and $\frac{1+\sqrt{d}}{2}$ for small values of $d$, the assertion can be easily verified for $D < 16$. So we assume $D > 16$, in which case (by lemma~\ref{lem_quotient_norm}) it suffices to show that $\nu_i > 1$ for $0 \leq i \leq l(\omega) - 2$.

By (\ref{eqn_total quotient to number}) one has $\omega = \frac{\alpha_{n+1}p_n + p_{n-1}}{\alpha_{n+1}q_n + q_{n-1}}$. Recall (\ref{eqn_nu_n}), proposition~\ref{prop_p q minus q p} and write
\begin{align*}
\alpha_{n+1} &= \frac{-\omega q_{n-1} + p_{n-1}}{\omega q_n - p_n} = \frac{(\omega q_{n-1} - p_{n-1})(\overline{p_n - \omega q_n})}{N(\xi_n)}\\
&= \frac{-N(\omega) q_{n-1}q_n - p_{n-1}p_n + p_n q_{n-1}\omega + p_{n-1}q_n \overline{\omega}}{(-1)^{n+1}\nu_n}\\
&=\begin{cases}
\frac{\frac{d-1}{4} q_{n-1}q_n - p_{n-1}p_n + p_{n-1}q_n + (-1)^{n+1}\omega}{(-1)^{n+1}\nu_n} & \text{if $\omega = \frac{1+\sqrt{d}}{2}$}\\
\frac{d q_{n-1}q_n - p_{n-1}p_n + (-1)^{n+1}\sqrt{d}}{(-1)^{n+1}\nu_n} & \text{if $\omega = \sqrt{d}$}.
\end{cases}
\end{align*}
Therefore, if $\nu_n = 1$ we have $\alpha_{n+1} \equiv \omega$ mod $1$, i.e.,
\begin{equation*}
    \alpha_{n+1} = [a_{n+1}, a_1, a_2, a_3,\cdots]
\end{equation*}
which implies that $n+1$ is a multiple of $l(\omega)$. Hence $\nu_i \geq 2$ for $i \leq l(\omega) -2$.
\end{proof}

For a symmetric sequence $\{ a_1, \cdots, a_{n} \}$ of positive integers, let $M_0$ be the 2 by 2 identity matrix and

\begin{equation*}
    M_n = \begin{pmatrix} a_1 &1\\ 1 &0 \end{pmatrix} \cdots \begin{pmatrix} a_n &1\\ 1 &0 \end{pmatrix} \quad \text{for $n \geq 1$}.
\end{equation*}
An induction argument easily proves that
\begin{equation*}
    M_n = \begin{pmatrix} q_n &q_{n-1}\\ r_n &r_{n-1} \end{pmatrix},\quad n\geq 0.
\end{equation*}

Here $\{ a_1, \cdots, a_{n} \}$ is symmetric, so $M_n$ is a symmetric matrix, i.e., $q_{n-1} = r_n$. Therefore $|M_n| = q_n r_{n-1} - q^2_{n-1} - = (-1)^n \equiv 1$ (mod 2) and $q_n q_{n-1} r_{n-1} \equiv 0$ (mod 2). Put
\begin{equation*}
    f(a_1,\cdots,a_n; T) = q^2_n T^2 + A T + B,\quad\text{where}
\end{equation*}
\begin{gather*}
    A = 4 q_{n-1} + 2 (-1)^n q_n q_{n-1}r_{n-1}, \quad B =  q^2_{n-1} r^2_{n-1} + 4 (-1)^n r^2_{n-1}.
\end{gather*}

Observe that $f(a_1,\cdots,a_n; T) \equiv q_n^2 T^2 + q_{n-1}^2 r_{n-1}^2$ (mod $4$).

\begin{thm}[Corollary 1 and 1A in \cite{Halter_Koch}]\label{thm_Halter_Koch}
Let $\{ a_1, \cdots, a_{n} \}$ be a symmetric sequence of positive integers and $q_n$, $q_{n-1}$, $r_{n-1}$ as above. The followings are equivalent:
\begin{enumerate}
  \item $\frak{D}(a_1,\cdots,a_n) \neq \emptyset$
  \item Either $q_n \equiv 1$ (mod 2) or $q_n \equiv q_{n-1}r_{n-1} \equiv 0$ (mod 2).
\end{enumerate}
In these cases, $\frak{D}(a_1,\cdots,a_n)$ consists of all $d > 0$ of the form
\begin{equation*}
    d = \frac{1}{4}f(T)
\end{equation*}
where $T$ is any integer satisfying $q_n T + (-1)^n q_{n-1} r_{n-1} > 0$, $f(T) \equiv 0$ mod $4$.

Likewise, the followings are equivalent:
\begin{enumerate}
  \item $\frak{D}'(a_1,\cdots,a_n) \neq \emptyset$
  \item Either $q_n \equiv 1$ (mod 2) or $q_n \equiv 1 + q_{n-1}r_{n-1} \equiv 0$ (mod 2).
\end{enumerate}
In these cases, $\frak{D}'(a_1,\cdots,a_n)$ consists of all $d > 0$, $d \equiv 1$ (mod 4) of the form
\begin{equation*}
    d = f(T)
\end{equation*}
where $T$ is any integer satisfying $q_n T + 1 + (-1)^n q_{n-1} r_{n-1} > 0$, $f(T) \equiv 1$ mod $4$.
\end{thm}

From now on we always assume $(p,q)=1$ whenever we write $p/q$ to denote a rational number. Recall that every irrational number has a unique continued fraction expansion. For convenience, write $x = [a_0,a_1,\cdots,a_m,\ast]$ if $[a_0,a_1,\cdots,a_m]$ is a convergent to $x$. We use the convention $[a_0,a_1,\cdots,a_m] = [a_0,a_1,\cdots,a_m,\infty]$. It is a simple observation that the set of positive real numbers is partitioned by the predecessors, i.e., for any positive integers $a_0,a_1,\cdots,a_m$ the set
\begin{equation*}
\{ x \in \mathbb{R} \mid x>0,\;\;x = [a_0,a_1,a_2,\cdots,a_m,\ast]\}
\end{equation*}
is a closed interval. The following proposition quantifies these intervals in view of $\omega_d$.

\begin{prop}\label{Portion with fixed Predecessors}
Let $a_1,a_2,\cdots,a_m$ be positive integers and $f(N)$ the number of non-square integers $d$ between 1 and $N$ such that
\begin{equation*}
\sqrt{d} = [a_0,a_1,a_2,\cdots,a_m,\ast]
\end{equation*}
for some $a_0$. Similarly let $f^{(1)}(N)$ be the number of non-square integers $d$ between 1 and $N$ which are congruent to 1 modulo 4 such that
\begin{equation*}
\omega_d = [a_0,a_1,a_2,\cdots,a_m,\ast]
\end{equation*}
for some $a_0$. Then
\begin{equation*}
    \lim_{N \rightarrow \infty} \frac{f(N)}{N} = \lim_{N \rightarrow \infty} \frac{4 f^{(1)}(N)}{N} = \frac{1}{q_m (q_m + q_{m-1})}.
\end{equation*}
\end{prop}

\begin{proof}[Sketch of proof]
This is almost trivial. Consider the numbers in the interval $(n^2,(n+1)^2)$ where $n$ is large. For $x$ in the range $n^2 < x < n^2 + 2n + 1$, the curve $y = \sqrt{x}$ may be approximated by a straight line of slope $\frac{1}{2n}$. The difference between $[n,a_1,a_2,\cdots,a_m]=[n,a_1,a_2,\cdots,a_m,\infty]$ and $[n,a_1,a_2,\cdots,a_m + 1]$ is $\frac{1}{q_m (q_m + q_{m-1})}$, so there are approximately $\frac{2n}{q_m (q_m + q_{m-1})}$ (non-square) integers between $n^2$ and $(n+1)^2$ that are counted by $f(N)$ for $N > (n+1)^2$. Similarly, for $(n-1)^2 < x < (n+1)^2$ consider the curve $y = \frac{1}{2}(1 + \sqrt{x})$ which is close to a straight line of slope $\frac{1}{4n}$. Extracting integers congruent to 1 modulo 4, we get the result.
\end{proof}

% -----------------------------------------------------------
\section{The Attached Intervals}\label{sec_Attached intervals}

Let $p/q$ be any rational number. Note that
\begin{equation*}
\left| \frac{p}{q} - \omega_d \right| = \begin{cases}
\frac{\nu}{(p+q (\omega_d-1))q}\;\; &\text{ if $d \equiv 1$ mod 4}\\
\frac{\nu}{(p+q \omega_d)q}\;\; &\text{ otherwise}
\end{cases}
\end{equation*}
where $\nu = |N(p - q \omega_d)|$. So we can interpret the norm of a quadratic integer $p - q \omega_d$ as a measure of how successful the approximation of $\omega_d$ by $\frac{p}{q}$ is. Proposition~\ref{prop_Hardy_convergent} shows that when
\begin{equation*}
\left| N(p - q \omega_d) \right| < \begin{cases}
\frac{p/q + \omega_d - 1}{2}\;\; &\text{ if $d \equiv 1$ mod 4}\\
\frac{p/q + \omega_d}{2}\;\; &\text{ otherwise},
\end{cases}
\end{equation*}
$p/q$ is a convergent to $\omega_d$. Lemma~\ref{lem_quotient_norm} also shows that $\xi_n$ becomes a quadratic unit if and only if the $(n+1)$-th convergent becomes as large as possible, namely $\alpha_{n+1} = \sqrt{D} -O(1)$. By (\ref{eqn_difference of convergents}) this means that $\omega_d = [a_0,\cdots,a_n,\alpha_{n+1}]$ is particularly close to $p_n/q_n$; in other words, $\nu_n = 1$ if the approximation of $\omega_d$ by $p_n/q_n$ is `sufficiently good'.

\begin{rem}\label{rem_period and symmetry}
By proposition~\ref{prop_partial quotients of omega d}, such thing happens if and only if $ n + 1 \equiv 0$ mod $l(\omega_d)$. Recall that $\omega_d = [a_0,\overline{a_1,\cdots,a_{l-1},a_l}]$ where $\{a_1,\cdots,a_{l-1}\}$ is symmetric. It follows that if a rational number $p/q$ satisfies $|N(p - q\omega_d)| = 1$ for some $d$, then $p/q - \lfloor p/q \rfloor = [0,a_1,\cdots,a_n]$ for some symmetric sequence $\{a_1,\cdots,a_n\}$.
\end{rem}

Based on the above context, for each rational number $p/q$ we assign tiny intervals $I^{i}_{p/q}$, $i=0,1$ that consist of points $x \in \mathbb{R}$ where $\sqrt{x}$ or $(1+\sqrt{x})/2$ is especially close to $p/q$. More specifically, we want these intervals to satisfy following property: whenever a non-square integer $d$ falls into that interval, the quality of approximation of $\omega_d$ by $p/q$ is sufficiently good and hence $p+q \omega_d$ or $p - q + q \omega_d$ becomes a quadratic unit. Explicitly, we build the intervals for a fixed $p/q \neq 1$ as follows.

Recall that there exists a unique sequence $\{ a_0,a_1,\cdots,a_m \}$ such that $p/q = [a_0,a_1,\cdots,$ $a_m,1] = [a_0,\cdots,a_{m-1}, a_m + 1]$ by proposition~\ref{prop_two ways of continued fractions}. From (\ref{eqn_difference of convergents}), for any positive number $\lambda$ one has
\begin{equation*}
\begin{cases}
[a_0,\cdots,a_m + 1, \lambda] < \frac{p}{q} < [a_0,\cdots,a_m,1,\lambda] \;\; &\text{ if $m$ is odd}\\
[a_0,\cdots,a_m,1, \lambda] < \frac{p}{q} < [a_0,\cdots,a_m + 1,\lambda] \;\; &\text{ if $m$ is even}.
\end{cases}
\end{equation*}

Since $\nu_n$ is an integer, by lemma~\ref{lem_quotient_norm} a total quotient that appears in the continued fraction expansion of $\omega_d$ cannot assume any values between $\sqrt{D}/2 + O(1)$ and $\sqrt{D} + O(1)$. Observe that $\sqrt{D}/2 = \omega_d + O(1) = a_0 +O(1)$. Unless $D$ is too small, therefore, we can say that $\nu_n=1$ if and only if $\alpha_{n+1} > \frac{4}{3} a_0 - O(1)$.
Let
\begin{gather*}
    A = A(p/q) = [a_0,\cdots,a_m +1,\frac{4}{3}a_0 -\frac{q_{m-1}}{q}],\\
    B = B(p/q) = [a_0,\cdots,a_m,1,\frac{4}{3}a_0 - \frac{q_m}{q}]
\end{gather*}
where $q_{m-1}$ and $q_m$ are the denominators of $[a_0,a_1,\cdots,a_{m-1}]$ and $[a_0,a_1,\cdots,$ $a_m]$ respectively. We take the intervals $I^{i}_{p/q}$ for $i = 0,1$ as
\begin{equation*}
I^0_{p/q} = \begin{cases}
\left(A^2,\; B^2 \right)\;\; &\text{ if $m$ is odd}\\
\left(B^2,\; A^2 \right)\;\; &\text{ if $m$ is even}
\end{cases}
\end{equation*}
and
\begin{equation*}
I^1_{p/q} = \begin{cases}
\left((2A-1)^2,\; (2B-1)^2 \right)\;\; &\text{ if $m$ is odd}\\
\left((2B-1)^2,\; (2A-1)^2 \right)\;\; &\text{ if $m$ is even.}
\end{cases}
\end{equation*}

We write
\begin{gather*}
    I^{0}_{p/q}\setminus \{ p^2/q^2 \} = I^{0,-}_{p/q} \bigcup I^{0,+}_{p/q},\quad
 I^{1}_{p/q}\setminus \{ (2p/q - 1)^2 \} = I^{1,-}_{p/q} \bigcup I^{1,+}_{p/q}
\end{gather*}

where $I^{i,-}_{p/q}$ (resp. $I^{i,+}_{p/q}$) is the left (resp. right) connected part of $I^{0}_{p/q}\setminus \{ p^2/q^2 \}$ or $I^{1}_{p/q}\setminus \{ (2p/q - 1)^2 \}$.
Denote the fractional part of $p/q$ by $\{p/q\} = k/q$. When $q\neq 1$, let $k^{-1}$ be the multiplicative inverse of $k$ modulo $q$.

\begin{thm}\label{thm_Interval contains an integer iff}
Let $p/q \geq 4$.
\begin{enumerate}
  \item $I^{0,+}_{p/q}$ contains an integer $d$ if and only if
  \begin{equation*}
  k^2 \equiv -1 \text{ (mod $q$)},\;2 a_0 \equiv k  \frac{1 +k^2}{q}  \text{ (mod $q$)}
  \end{equation*}

  \item $I^{0,-}_{p/q}$ contains an integer $d$ if and only if
  \begin{equation*}
  k^2 \equiv 1 \text{ (mod $q$)},\;2 a_0 \equiv k  \frac{1 -k^2}{q}  \text{ (mod $q$)}
  \end{equation*}

  \item $I^{1,+}_{p/q}$ contains an integer $d \equiv 1$ mod 4 if and only if
  \begin{equation*}
  k^2 \equiv -1 \text{ (mod $q$)},\;2 a_0 \equiv 1 + k  \frac{1 +k^2}{q}  \text{ (mod $q$)}
  \end{equation*}

  \item $I^{1,-}_{p/q}$ contains an integer $d \equiv 1$ mod 4 if and only if
  \begin{equation*}
  k^2 \equiv 1 \text{ (mod $q$)},\;2 a_0 \equiv 1 + k  \frac{1 -k^2}{q}  \text{ (mod $q$)}.
  \end{equation*}
\end{enumerate}

In each case the integer $d$ is given by
\begin{equation}\label{eqn_d contained in the intervals}
d =\begin{cases}
a_0^2 + \frac{2k}{q}a_0 + \frac{k^2 + 1}{q^2} & \text{in case (1)}\\
a_0^2 + \frac{2k}{q}a_0 + \frac{k^2 - 1}{q^2} & \text{in case (2)}\\
(2a_0 -1)^2 + \frac{4k}{q}(2 a_0 - 1) + \frac{4k^2 + 4}{q^2} & \text{in case (3)}\\
(2a_0 -1)^2 + \frac{4k}{q}(2 a_0 - 1) + \frac{4k^2 - 4}{q^2} & \text{in case (4)}
\end{cases}
\end{equation}
and
\begin{equation}\label{eqn_norm of p minus q sqrt_d}
N(p-q\sqrt{d}) = \begin{cases} -1 & \text{in case (1)}\\
1 & \text{in case (2)}
\end{cases}
\end{equation}

\begin{equation}\label{eqn_norm of p minus q omega_d}
N\left(p-q\frac{1+\sqrt{d}}{2}\right) = \begin{cases} -1 & \text{in case (3)}\\
1 & \text{in case (4)}
\end{cases}
\end{equation}
\end{thm}

\begin{proof}

When $q = 1$ we have $p/q = [p]=[p-1,1]$, so $m=0$, $a_0 = p-1$, $q_0 = 1$, $q_{-1} = 0$. Since $m$ is even, $I^{0,-}_{p/q} = (B^2,p^2/q^2)$ and $I^{0,+}_{p/q} = (p^2/q^2, A^2)$. In this case
\begin{align*}
    |I^{0,+}_{p/q}| &= \left[p,\frac{4}{3}(p-1)\right]^2 - p^2 = \left( p + \frac{3}{4(p-1)} \right)^2 - p^2 \\
&= \frac{3}{4(p-1)}\left( 2p + \frac{3}{4(p-1)} \right) \\
&= \frac{3}{2} + \frac{3}{2(p-1)} + \frac{9}{16(p-1)^2}
\end{align*}
which is less than 2 for $p \geq 5$. Similarly,
\begin{align*}
|I^{0,-}_{p/q}| &= p^2 - \left[p-1,1,\frac{4}{3}(p-1)-1\right]^2 = p^2 - \left( p-1 + \frac{1}{1 + \frac{3}{4p-7}} \right)^2\\
&= \frac{3}{2} + \frac{3}{2(p-1)} - \frac{9}{16(p-1)^2} < 2.
\end{align*}
The integers contained in $I^{0}_{p/q} \setminus\{ (p/q)^2 \}$ are therefore $p^2 \pm 1$. In the same manner $|I^{1,+}_{p/q}| = 6 + \frac{3}{p-1} + \frac{9}{4(p-1)^2}$, $|I^{1,-}_{p/q}| = 6 + \frac{3}{p-1} - \frac{9}{4(p-1)^2}$ which are less than 8, and the integers $\equiv 1$ mod 4 contained in $I^{1}_{p/q} \setminus \{ (2p/q -1)^2 \}$ are $(2p-1)^2 \pm 4$. All the assertions become trivial in this case, so we assume $q \neq 1$ and write $p/q = a_0 + k/q$.

Put
\begin{gather*}
    \alpha_1 = [a_1,\cdots,a_m + 1],\quad \beta_1 = \left[a_1,\cdots,a_m+1,\frac{4}{3}a_0 -\frac{q_{m-1}}{q}\right],\\
    \gamma_1 = \left[a_1,\cdots,a_m,1,\frac{4}{3}a_0 - \frac{q_m}{q}\right]
\end{gather*}
so that $p/q = a_0 + 1/\alpha_1$, $A = a_0 + 1/\beta_1$ and $B = a_0 + 1/\gamma_1$.

Assume $m$ is even. By (\ref{eqn_difference of convergents})
\begin{equation*}
    A = \frac{p}{q} + \frac{3}{4 a_0 q^2},\quad B = \frac{p}{q} - \frac{3}{4 a_0 q^2}
\end{equation*}
and
\begin{equation*}
    A^2 - \frac{p^2}{q^2} = \left( A + \frac{p}{q} \right) \frac{3}{4 a_0 q^2} = \left( 2 a_0 + \frac{1}{\alpha_1} +\frac{1}{\beta_1} \right) \frac{3}{4 a_0 q^2} = \frac{3}{2 q^2} \left( 1 + \frac{\alpha_1 + \beta_1}{2 a_0 \alpha_1 \beta_1} \right),
\end{equation*}
\begin{equation*}
    \frac{p^2}{q^2} - B^2 = \left( \frac{p}{q} + B \right) \frac{3}{4 a_0 q^2} = \left( 2 a_0 + \frac{1}{\alpha_1} +\frac{1}{\gamma_1} \right) \frac{3}{4 a_0 q^2} = \frac{3}{2 q^2} \left( 1 + \frac{\alpha_1 + \gamma_1}{2 a_0 \alpha_1 \gamma_1} \right).
\end{equation*}
Here $\alpha_1, \beta_1, \gamma_1 \geq 1$, so
\begin{gather*}
    1 < 1 + \frac{\alpha_1 + \beta_1}{2 a_0 \alpha_1 \beta_1} \leq 1 + \frac{1}{a_0},\qquad 1 < 1 + \frac{\alpha_1 + \gamma_1}{2 a_0 \alpha_1 \gamma_1}  \leq 1 + \frac{1}{a_0}.
\end{gather*}
Since $a_0 > 3$,
\begin{equation*}
    \frac{3}{2} \left( 1 + \frac{1}{a_0} \right) < 2
\end{equation*}
and the lengths of the intervals $I^{0,+}_{p/q}$, $I^{0,-}_{p/q}$ are between $1/q^2$ and $2/q^2$. It follows that
\begin{align*}
& I^{0,+}_{p/q} \bigcap \mathbb{Z} \neq \emptyset\\
&\Leftrightarrow \left\lceil \frac{p^2}{q^2} \right\rceil = \frac{p^2+1}{q^2}\\
&\Leftrightarrow \left( a_0 + \frac{k}{q} \right)^2 \equiv -\frac{1}{q^2} \text{ (mod 1)}\\
&\Leftrightarrow 2 a_0 \frac{k}{q} \equiv -\frac{k^2}{q^2} - \frac{1}{q^2} \text{ (mod 1)}\\
&\Leftrightarrow 2 a_0 k \equiv -\frac{k^2 + 1}{q} \text{ (mod $q$)}\\
\end{align*}
But $2 a_0 k$ is an integer, so the last congruence is equivalent to
\begin{align*}
k^2 \equiv -1 \text{ (mod $q$)},\;2 a_0 \equiv k^{-1}\left(  \frac{-1 -k^2}{q}  \right) \equiv k\left(  \frac{1 +k^2}{q}  \right)\text{ (mod $q$)}.
\end{align*}

Similarly, $I^{0,-}_{p/q} \bigcap \mathbb{Z} \neq \emptyset \Leftrightarrow \left\lfloor \frac{p^2}{q^2} \right\rfloor = \frac{p^2 - 1}{q^2}$ which is equivalent to
\begin{equation*}
    k^2 \equiv 1 \text{ (mod $q$)},\;2 a_0 \equiv k^{-1}\left(  \frac{1 -k^2}{q}  \right) \equiv k\left(  \frac{1 - k^2}{q}  \right)\text{ (mod $q$)}.
\end{equation*}

When $m$ is odd, the intervals become $I^{0,-}_{p/q} = (A^2,p^2/q^2)$ and $I^{0,+}_{p/q} = (p^2/q^2,B^2)$ and we get the same conclusions via the same computation. This proves (1) and (2).

Assume $m$ is even again. As for $I^{1}_{p/q}$,
\begin{multline*}
    |I^{1,+}_{p/q}| = (2A-1)^2 - \left( 2 \frac{p}{q} - 1 \right)^2 = \left( 2A + 2\frac{p}{q} -2 \right) \frac{3}{2 a_0 q^2}\\
     = \left( 2 a_0 - 1 + \frac{1}{\alpha_1} +\frac{1}{\beta_1} \right) \frac{3}{a_0 q^2} \\
     = \frac{6}{q^2} \cdot \left( 1 - \frac{1}{2 a_0} + \frac{\alpha_1 + \beta_1}{2 a_0 \alpha_1 \beta_1} \right)
\end{multline*}
and
\begin{multline*}
    |I^{1,-}_{p/q}| = \left( 2 \frac{p}{q} - 1 \right)^2 - (2B-1)^2 = \left( 2\frac{p}{q} +2B-2 \right) \frac{3}{2 a_0 q^2} \\
    = \frac{6}{q^2} \cdot \left( 1 - \frac{1}{2 a_0} + \frac{\alpha_1 + \gamma_1}{2 a_0 \alpha_1 \gamma_1} \right).
\end{multline*}
Here $\alpha_1,\beta_1,\gamma_1 \geq 1$ and $a_0 \geq 4$, so
\begin{equation*}
    \frac{7}{8} \leq 1 - \frac{1}{2 a_0} <  1 - \frac{1}{2 a_0} + \frac{\alpha_1 + \beta_1}{2 a_0 \alpha_1 \beta_1} \leq 1 + \frac{1}{2 a_0} \leq \frac{9}{8}
\end{equation*}
and hence $|I^{1,-}_{p/q}|,|I^{1,+}_{p/q}| \in \left[ \frac{21}{4 q^2}, \frac{27}{4 q^2} \right]$.
For $\theta,\eta \in \mathbb{R}$, $\theta \neq 0$, let $\theta  I^{1,+}_{p/q} + \eta = \{x \in \mathbb{R}\;|\; (x-\eta) / \theta \in I^{1,+}_{p/q} \}$. Denoting the set of integers congruent to 1 mod 4 by $4 \mathbb{Z} + 1$, we have
\begin{align*}
& I^{1,+}_{p/q} \bigcap (4 \mathbb{Z} + 1) \neq \emptyset\\
&\Leftrightarrow \left( \frac{1}{4}I^{1,+}_{p/q} - \frac{1}{4} \right) \bigcap \mathbb{Z} \neq \emptyset
\end{align*}
But $\frac{1}{4} I^{1,+}_{p/q}$ has length $t / q^2$ for some $t$ satisfying $1 < 21/16 \leq  t \leq 27/16 < 2$. Writing $\frac{1}{4} (2p/q-1)^2 - \frac{1}{4} = p^2/q^2 - p/q$, we thus have
\begin{align*}
& I^{1,+}_{p/q} \bigcap (4 \mathbb{Z} + 1) \neq \emptyset\\
&\Leftrightarrow \left\lceil \frac{1}{4} \left( 2\frac{p}{q} - 1\right)^2 - \frac{1}{4} \right\rceil = \frac{p^2 - pq +1}{q^2}\\
&\Leftrightarrow \left( a_0 + \frac{k}{q} \right)^2 - \left( a_0 + \frac{k}{q} \right) \equiv -\frac{1}{q^2} \text{ (mod 1)}\\
&\Leftrightarrow  2 a_0 k - k \equiv -\frac{k^2 + 1}{q} \text{ (mod $q$)}\\
&\Leftrightarrow  k^2 \equiv -1 \text{ (mod $q$)},\;2 a_0 \equiv 1 + k^{-1}\left( - \frac{1 + k^2}{q}  \right) \equiv 1 + k \left(  \frac{1 + k^2}{q}  \right)\text{ (mod $q$)}.
\end{align*}
Similarly we have
\begin{align*}
& I^{1,-}_{p/q} \bigcap (4 \mathbb{Z} + 1) \neq \emptyset\\
&\Leftrightarrow  k^2 \equiv 1 \text{ (mod $q$)},\;2 a_0 \equiv  1 + k \left(  \frac{1 - k^2}{q}  \right)\text{ (mod $q$)}.
\end{align*}

Odd $m$ gives exactly the same congruences as these, which proves (3) and (4).

Note that in each case from (1) to (4), $d$ should be the integer closest to $p^2/q^2$ or the integer $\equiv 1$ mod 4 closest to $(2 p/q - 1)^2$. Along the proof we already showed that this integer is $\frac{p^2+1}{q^2}$, $\frac{p^2-1}{q^2}$, $4 \frac{p^2 - pq + 1}{q^2} + 1$ and $4 \frac{p^2 - pq - 1}{q^2} + 1$ in each case, whence (\ref{eqn_d contained in the intervals}) follows.

Finally, when (1) or (2), writing $p = a_0 q + k$
\begin{equation*}
    N(p-q\sqrt{d}) = (a_0 q + k)^2 - q^2\left(a_0^2 + \frac{2k}{q}a_0 + \frac{k^2 \pm 1}{q^2}\right) = \mp 1.
\end{equation*}
Similarly when (3) or (4)
\begin{equation*}
    N\left(p-q\frac{1+\sqrt{d}}{2}\right) = p^2 - pq + \frac{q^2}{4} - \frac{q^2}{4}d = \mp 1.
\end{equation*}

 This completes the proof.
\end{proof}

%-----------------------------------------------------------------------------
\section{Dominance of the least elements}\label{sec_Dominance of the least elements}

Let
\begin{align*}
&\frak{I}^{o,+} = \{\;(y,x) \in \mathbb{Z}^2\;|\; 0 \leq x < y,\; x^2 \equiv 1\; \text{mod $y$, $y$ is odd} \},\\
&\frak{I}^{o,-} = \{\;(y,x) \in \mathbb{Z}^2\;|\; 0 \leq x < y,\; x^2 \equiv -1\; \text{mod $y$, $y$ is odd} \},\\
&\frak{I}^{e,+} = \{\;(y,x) \in \mathbb{Z}^2\;|\; 0 \leq x < y,\; x^2 \equiv 1\; \text{mod $y$, $y$ is even} \},\\
&\frak{I}^{e,-} = \{\;(y,x) \in \mathbb{Z}^2\;|\; 0 \leq x < y,\; x^2 \equiv -1\; \text{mod $y$, $y$ is even} \}
\end{align*}
and $\frak{I}$ the union of these four sets. Put
\begin{equation*}
\tilde{y} = \begin{cases}
\frac{y}{2}\;\; &\text{ if $y$ is even}\\
y\;\; &\text{ otherwise}.
\end{cases}
\end{equation*}

Assume $(y,x) \in \frak{I}^{e,-}$. If $\frac{1 + x^2}{y}$ is even (resp. odd) then there exists $a_0$ satisfying $2a_0 \equiv x \frac{1 + x^2}{y}$ mod $y$ (resp. $\equiv 1 + x \frac{1 + x^2}{y}$ mod $y$). In this case theorem~\ref{thm_Interval contains an integer iff} (1) (resp.(3)) gives an arithmetic progression of $a_0 \geq 5$ with common difference $y/2 = \tilde{y}$. When $(y,x) \in \frak{I}^{o,-}$, both of (1) and (3) give such arithmetic progressions with common difference $y = \tilde{y}$. Similar things can be said about $\frak{I}^{e,+}$, $\frak{I}^{o,+}$, and (\ref{eqn_d contained in the intervals}) gives a quadratic progression for each case.

In theorem~\ref{thm_Interval contains an integer iff} we assumed $p/q \geq 5$, but now we extend the arithmetic progressions to the range $a_0 > 0$. Observe that (\ref{eqn_norm of p minus q sqrt_d}) and (\ref{eqn_norm of p minus q omega_d}) are still valid for this range, once $d$ is determined by (\ref{eqn_d contained in the intervals}); therefore (\ref{eqn_d contained in the intervals}) gives quadratic progressions of positive integers for each pair $(y,x) \in \frak{I}$.

Theorem~\ref{thm_Halter_Koch} specified a quadratic progression for a symmetric sequence, so we can easily compare theorem~\ref{thm_Halter_Koch} and \ref{thm_Interval contains an integer iff} now. Let $\frak{D}^i(y,x)$ be the set of non-square integers $d$ given by (\ref{eqn_d contained in the intervals}) (case (1) or (2) for $i = 0$; case (3) or (4) for $i=1$) in which $a_0 > 0$ runs through corresponding arithmetic progressions mentioned above.

\begin{prop}\label{prop_y_x and symmetric sequence}
Let $y > 0$, $0 \leq x \leq y$, $(x,y) = 1$. Then the followings are equivalent:
\begin{enumerate}
  \item $x^2 \equiv -1$ mod $y$ (resp. $\equiv 1$ mod $y$)
  \item $\frac{x}{y} = [a_0,a_1,\cdots,a_n]$ for some symmetric sequence $\{a_1,\cdots,a_n\}$ of positive integers where $n$ is even (resp. $n$ is odd). Here, $n=0$ corresponds to the empty sequence.
\end{enumerate}
If this holds and $y$ is even, let $t = \frac{x^2 + 1}{y}$ (resp. $=\frac{x^2 - 1}{y}$). Then $t \equiv q_{n-1} r_{n-1}$ mod $2$.
\end{prop}

\begin{proof}
Since $1 = [1] = [0,1]$, the assertion is trivial for $x = y = 1$. Hence we assume $x \neq y$.

Assume $(y,x)\in \frak{I}$. By theorem~\ref{thm_Interval contains an integer iff}, $\frak{D}^i(y,x) \neq \emptyset$ for at least one $i$. Therefore by remark \ref{rem_period and symmetry}, $x/y = [0,a_1,\cdots,a_n]$ for some symmetric sequence $\{a_1,\cdots,a_n\}$. Here, by (\ref{eqn_nu_n}) and (\ref{eqn_norm of p minus q sqrt_d}), (\ref{eqn_norm of p minus q omega_d}), $n$ is even if and only if $x^2 \equiv -1$ mod $y$.

Conversely, assume $\frac{x}{y} = [0,a_1,\cdots,a_n]$ for some symmetric sequence $\{a_1,$ $\cdots,$ $a_n\}$. By theorem~\ref{thm_Halter_Koch} there exists a large $d$ such that
\begin{equation*}
    \sqrt{d} = [\lfloor \sqrt{d} \rfloor,\overline{ a_1,\cdots,a_n,2\lfloor \sqrt{d} \rfloor}]\quad  \text{or}\quad \omega_d = [\lfloor \omega_d \rfloor, \overline{a_1,\cdots,a_n,2\lfloor \omega_d \rfloor - 1}].
\end{equation*}
 For such $d$, $\nu_n = 1$ by lemma~\ref{lem_quotient_norm}. But $p_n/q_n = \lfloor \sqrt{d} \rfloor + x/y$ (or $=\lfloor \omega_d \rfloor + x/y$) here, and $\nu_n = (-1)^{n+1}(p_n^2 - q_n^2 d)$ (or $(-1)^{n+1}(p_n^2 - p_n q_n + q_n^2 \frac{1 - d}{4})$) by (\ref{eqn_nu_n}) and (\ref{eqn_xi_n}). Hence $x^2 \equiv (-1)^{n+1}$ mod $y$.

Now assume $y$ is even. Comparing to theorem~\ref{thm_Halter_Koch}, one easily has
\begin{gather*}
\text{$t$ is even} \;\Leftrightarrow\; \frak{D}^0(y,x) \neq \emptyset \;\Leftrightarrow\; \frak{D}(a_1,\cdots,a_n) \neq \emptyset \;\Leftrightarrow\; \text{$q_{n-1} r_{n-1}$ is even}
\end{gather*}
which completes the proof.
\end{proof}

Let $x/y = [0,a_1,\cdots,a_n]$ where $\{ a_1,\cdots,a_n \}$ is symmetric. For any $d \in \frak{D}^0(y,x)$ (resp. $d \in \frak{D}^1(y,x)$) we have
\begin{gather*}
\sqrt{d} = [\lfloor \sqrt{d} \rfloor, \overline{a_1,\cdots,a_n,2\lfloor \sqrt{d} \rfloor}]\quad  (\text{ resp. }\omega_d = [\lfloor \omega_d \rfloor, \overline{a_1,\cdots,a_n,2\lfloor \omega_d \rfloor - 1}])
\end{gather*}
and unless $2 \lfloor \sqrt{d} \rfloor$ (resp. $2\lfloor \omega_d \rfloor - 1$) appears in $\{a_1,a_2,\cdots,a_n\}$, the period $l(\sqrt{d})$ (resp. $l(\omega_d)$) is exactly $n+1$. Since the partial quotients of $\sqrt{d}$ (resp. $\omega_d$) cannot exceed $2 \lfloor \sqrt{d} \rfloor$ (resp. $2\lfloor \omega_d \rfloor - 1$), such exceptional case (i.e., $l(\sqrt{d})$ or $l(\omega_d) < n+1$) may possibly occur only when $d$ is the least element of $\frak{D}^i(y,x)$. Let $\overline{\frak{D}^i}(y,x)$ be the set $\frak{D}^i(y,x)$ where this possible exception is removed, i.e., with the least element discarded if its period is less than $n+1$. Then for each non-square integer $d$ (resp. non-square integer $d \equiv 1$ mod $4$), there is a unique $(y,x)\in\frak{I}$ such that $d \in \overline{\frak{D}^0}(y,x)$ (resp. $d \in \overline{\frak{D}^1}(y,x)$).

As mentioned in the introduction, when $d \in \overline{\frak{D}^0}(y,x)$ is not the least element of $\overline{\frak{D}^0}(y,x)$, some nice assertions like Ankeny-Artin-Chowla conjecture become true for $d$. This is because $\varepsilon_d$ is relatively small. By proposition~\ref{prop_partial quotients of omega d} the fundamental unit of $\mathbb{Z}[\sqrt{d}]$ is $\varepsilon_d = \lfloor \sqrt{d} \rfloor y + x + y \sqrt{d}$, and since $\frac{\lfloor \sqrt{d} \rfloor y + x}{y}$ is a convergent to $\sqrt{d}$, $\varepsilon_d = 2 y \sqrt{d} + O(1/y)$. But $d$ is not the smallest element in $\overline{\frak{D}^0}(y,x)$, so $\lfloor \sqrt{d} \rfloor = a_0 > \tilde{y}$ in (\ref{eqn_d contained in the intervals}) and hence $d > \tilde{y}^2$. In other words, $\varepsilon_d \ll d$ and by (\ref{eqn_Dirichlet class number formula}) $h_d \gg D^{1/2 - \epsilon}$.

We will call a non-square integer $d$ a \emph{least to $i$} if $d$ is the smallest element of $\overline{\frak{D}^i}(y,x)$ for some $(y,x)\in\frak{I}$. Let $\overline{\frak{D}^i}$ be the set of all non-square positive integers that are the leasts to $i$. It is a general belief that the class number is usually very small, say, $h_d \ll \log^2 D$ on average (for example, see conjecture 7 in \cite{Hooley_1}). Such a strong assertion is out of reach at this moment, but we can at least show that almost all non-square positive integers (resp. integers $\equiv 1$ mod $4$) are the leasts to $0$ (resp. leasts to $1$).

\begin{thm}\label{thm_dominance of least elements}
\begin{equation*}
    \sum_{d \in \overline{\frak{D}^0}} \frac{1}{d^s} \approx \zeta(s),\quad \sum_{d \in \overline{\frak{D}^1}} \frac{1}{d^s} \approx \sum_{\substack{d \equiv 1\text{ mod $4$}\\ d \text{ is non-square} }} \frac{1}{d^s} \qquad \text{ as $s \rightarrow 1+$}
\end{equation*}
\end{thm}
(where `$\approx$' means the difference is bounded.)

\begin{proof}
We assume $y > 0$. Let
\begin{align*}
    &V(y) = \{ x \mid 0 \leq x < y,\; x^2 \equiv \pm 1\text{ (mod $y$)}  \},\\
    &I_1 = \{(y,x) \in \frak{I} \;|\; \frak{D}^0(y,x) \neq \emptyset,\; \overline{\frak{D}^0}(y,x) = \frak{D}^0(y,x) \},\\
    &I_2 = \{(y,x) \in \frak{I} \;|\; \frak{D}^0(y,x) \neq \emptyset,\; \overline{\frak{D}^0}(y,x) \neq \frak{D}^0(y,x) \}.
\end{align*}

For $(y,x) \in I_1 \bigcup I_2$, let $\frak{a} = \frak{a}(y,x)$ be the least positive integer $t$ satisfying $2t \equiv x\frac{1\pm x^2}{y}$ mod $y$ (where $x^2 \equiv \mp 1$ mod $y$). According to (\ref{eqn_d contained in the intervals}), write
\begin{align*}
    d = d(y,x;k) &= a_0^2 + \frac{2x}{y}a_0 + \frac{x^2\pm 1}{y^2}\\
&= (\frak{a} + \tilde{y} k)^2 + \frac{2x}{y}(\frak{a} + \tilde{y} k) + \frac{x^2\pm 1}{y^2}\\
&= \left( \frak{a} + \frac{x}{y} + \tilde{y} k \right)^2 \pm \frac{1}{y^2}.
\end{align*}

Assume $s>1$. We can easily compute the following sums:
\begin{align*}
    \zeta(s) - \zeta(2s) &= \sum_{y=1}^{\infty} \sum_{x \in V(y)} \sum_{d \in {\overline{\frak{D}^0}(y,x)}} \frac{1}{d^s}\\
&= \sum_{(y,x) \in I_1} \left(  \frac{1}{d(y,x;0)^s}+ \sum_{k=1}^{\infty} \frac{1}{\left( (\frak{a} + x/y + \tilde{y} k)^2 + O(1/y^2)  \right)^s} \right)\\
&\qquad +\sum_{(y,x) \in I_2}\sum_{k=1}^{\infty}\frac{1}{\left( (\frak{a} + x/y + \tilde{y} k)^2 + O(1/y^2)  \right)^s}\\
&= \sum_{(y,x) \in I_1} \left( \frac{1}{d(y,x;0)^s} + \frac{1}{\tilde{y}^{2s}}\left(\zeta(2s) - O(1)\right) \right)\\
&\qquad + \sum_{(y,x) \in I_2} \frac{1}{\tilde{y}^{2s}}\left(  \zeta(2s) - O(1)  \right).
\end{align*}

Now consider
\begin{align*}
    &V^+(y) = \{ x \mid 0 \leq x < y,\; x^2 \equiv 1\text{ (mod $y$)}  \},\\
    &V^-(y) = \{ x \mid 0 \leq x < y,\; x^2 \equiv -1\text{ (mod $y$)} \}.
\end{align*}

Using Chinese remainder theorem and the fact that the group of units modulo $p^n$ for an odd prime $p$ is cyclic, it is easy to see that  $x^2 \equiv 1$ (mod $y$) has $O(2^{\omega(y)})$ solutions where $\omega(y)$ is the number of distinct prime factors of $y$. The same is true for $x^2 \equiv -1$ (mod $y$) if $-1$ is a quadratic residue for every prime divisor of $y$; otherwise it has no solutions. The Euler product form of zeta function gives $\sum_{n=1}^{\infty} \frac{2^{\omega(n)}}{n^u} = \frac{\zeta(u)^2}{\zeta(2u)}$ for $u > 1$ (see theorem 301 of \cite{Hardy}), whence
\begin{align*}
\sum_{(y,x)\in I_1\bigcup I_2} \frac{1}{\tilde{y}^{2s}}(\zeta(2s)-O(1))
&\ll \sum_{y=1}^{\infty} \sum_{x \in V(y)} \frac{1}{y^{2s}}\\
&\ll \sum_{y=1}^{\infty} \sum_{x \in V^+(y)} \frac{1}{y^{2s}}\\
&\ll \sum_{y=1}^{\infty} \frac{2^{\omega(y)}}{y^{2s}} = \frac{\zeta(2s)^2}{\zeta(4s)} = O(1)
\end{align*}
and therefore
\begin{equation*}
    \sum_{(y,x)\in I_1} \frac{1}{d(y,x;0)^s} \approx \zeta(s)\qquad \text{as $s \rightarrow 1+$}.
\end{equation*}

Similarly, let
\begin{align*}
    &I'_1 = \{(y,x) \in \frak{I} \;|\; \frak{D}^1(y,x) \neq \emptyset,\; \overline{\frak{D}^1}(y,x) = \frak{D}^1(y,x) \},\\
    &I'_2 = \{(y,x) \in \frak{I} \;|\; \frak{D}^1(y,x) \neq \emptyset,\; \overline{\frak{D}^1}(y,x) \neq \frak{D}^1(y,x) \}
\end{align*}
and for $(y,x) \in I'_1 \bigcup I'_2$, let $\frak{a}'$ be the least positive integer $t$ satisfying $2t \equiv 1 + x\frac{1\pm x^2}{y}$ mod $y$ (where $x^2 \equiv \mp 1$ mod $y$). Write
\begin{align*}
    d' = d'(y,x;k) = \left( 2 \frak{a}' - 1 + \frac{2x}{y} + 2 \tilde{y} k \right)^2 \pm \frac{4}{y^2}.
\end{align*}

Like before,
\begin{align*}
&\sum_{\substack{d \equiv 1\text{ mod $4$}\\ d \text{ is non-square} }} \frac{1}{d^s}\\
&= \sum_{n=0}^{\infty} \frac{1}{(4n+1)^s} - \sum_{m=0}^{\infty} \frac{1}{(2m+1)^{2s}}\\
&= \sum_{y=1}^{\infty} \sum_{x \in V(y)} \sum_{d \in {\overline{\frak{D}^1}(y,x)}} \frac{1}{d^s}\\
&= \sum_{(y,x) \in I'_1} \left(  \frac{1}{d'(y,x;0)^s}+ \sum_{k=1}^{\infty} \frac{1}{\left( (2\frak{a}'-1 + 2x/y + 2 \tilde{y} k)^2 + O(1/y^2)  \right)^s} \right)\\
&\qquad \qquad +\sum_{(y,x) \in I'_2}\sum_{k=1}^{\infty}\frac{1}{\left( (2\frak{a}'-1 + 2x/y + 2\tilde{y} k)^2 + O(1/y^2)  \right)^s}\\
&= \sum_{(y,x) \in I'_1} \left( \frac{1}{d'(y,x;0)^s} + \frac{1}{(2\tilde{y})^{2s}}\left(\zeta(2s) - O(1)\right) \right)\\
&\qquad \qquad + \sum_{(y,x) \in I'_2} \frac{1}{(2\tilde{y})^{2s}}\left(  \zeta(2s) - O(1)  \right)
\end{align*}
and again
\begin{equation*}
    \sum_{\substack{d \equiv 1\text{ mod $4$}\\ d \text{ is non-square} }} \frac{1}{d^s} \approx \sum_{(y,x) \in I'_1} \frac{1}{d'(y,x;0)^s}\qquad \text{as $s \rightarrow 1+$}.
\end{equation*}
\end{proof}

We can apply this to real quadratic fields as follows. Assume $d$ is square-free. We say that $\mathbb{Q}(\sqrt{d})$ is of \emph{the least type} if either $d \equiv 2,3$ mod $4$ and $d \in \overline{\frak{D}^0}$, or $d \equiv 1$ mod $4$ and $d \in \overline{\frak{D}^1}$.
\begin{cor}
Let $S(X)$ be the set of square-free integers between $1$ and $X$, and $S(X;c,k)$ as in theorem~\ref{thm_square-free_congruence}. Then
\begin{align*}
    \frac{| S(X) \bigcap \overline{\frak{D}^0}|}{|S(X)|} \sim  \frac{|S(X) \bigcap \overline{\frak{D}^1}|}{S(X;1,4)} \sim 1\qquad (X \rightarrow \infty).
\end{align*}
In particular, almost all real quadratic number fields are of the least type.
\end{cor}
\begin{proof}
By theorem~\ref{thm_dominance of least elements}, the ratio of non-square integers that are not the leasts to $0$ or $1$ is asymptotically zero. Since the square-free integers congruent to $1,2$ and $3$ mod $4$ constitute positive density sets by theorem~\ref{thm_square-free_congruence}, the corollary follows.
\end{proof}

\bibliographystyle{spmpsci}      % mathematics and physical sciences
\bibliography{NRQFL_Ramanujan}   % name your BibTeX data base

\begin{thebibliography}{10}
\providecommand{\url}[1]{{#1}}
\providecommand{\urlprefix}{URL }
\expandafter\ifx\csname urlstyle\endcsname\relax
  \providecommand{\doi}[1]{DOI~\discretionary{}{}{}#1}\else
  \providecommand{\doi}{DOI~\discretionary{}{}{}\begingroup
  \urlstyle{rm}\Url}\fi

\bibitem{AnkenyArtinChowla}
Ankeny, N., Artin, E., Chowla, S.: The class-number of real quadratic number
  fields.
\newblock Ann. of Math.(2) \textbf{56}(3), 479--493 (1952)

\bibitem{Chao_Hua}
Chao-Hua, J.: The distribution of square-free numbers.
\newblock Sci. China Ser. A \textbf{36}(2), 154--169 (1993)

\bibitem{Friesen}
Friesen, C.: On continued fractions of given period.
\newblock In: Proc. Amer. Math. Soc, vol. 103, pp. 9--14 (1988)

\bibitem{Halter_Koch}
Halter-Koch, F.: Continued fractions of given symmetric period.
\newblock Fibonacci Quart. (29), 298--303 (1991)

\bibitem{Hardy}
Hardy, G., Wright, E.: An introduction to the theory of numbers, 5th edn.
\newblock Clarendon Press. Oxford (1979)

\bibitem{Hashimoto2001143}
Hashimoto, R.: Ankeny-\text{A}rtin-\text{C}howla conjecture and continued
  fraction expansion.
\newblock Journal of Number Theory \textbf{90}(1), 143 -- 153 (2001)

\bibitem{Hooley_1}
Hooley, C.: On the \text{P}ellian equation and the class number of indefinite
  binary quadratic forms.
\newblock J. Reine Angew. Math. \textbf{353}(2), 98--131 (1984)

\bibitem{Kawamoto}
Kawamoto, F., Tomita, K.: Continued fractions and certain real quadratic fields
  of minimal type.
\newblock J. Math. Soc. Japan \textbf{60}(3), 865--903 (2008)

\bibitem{Li}
Li, X.: Upper bounds on {L}-functions at the edge of the critical strip
  \textbf{4}, 727--755 (2010)

\bibitem{niven1991introduction}
Niven, I., Zuckerman, H., Montgomery, H.: An introduction to the theory of
  numbers, 5th edn.
\newblock Wiley (1991)

\bibitem{Prachar}
Prachar, K.: \"{U}ber die kleinste quadratfreie zahl einer arithmetischen
  reihe.
\newblock Monatsh.Math. \textbf{62}, 173--176 (1958)

\end{thebibliography}

\end{document}